\documentclass[reqno,11pt,centertags]{amsart}
\usepackage{amsmath,amsthm,amscd,amssymb,latexsym,upref,stmaryrd}
\usepackage{color}
\newcommand{\RR}{{\mathbb R}}

\newcommand{\CC}{{\mathbb C}}

\newcommand{\ou}{{\overline{u}}}
\newcommand{\pa}{{\partial}}
\newcommand{\rmi}{{\mathrm{i}}}

\newcommand{\bH}{{\mathbb{H}}}
\newcommand{\bY}{{\mathbb{Y}}}
\newcommand{\rmd}{{\mathrm{d}}}
\newcommand{\rms}{{\mathrm{s}}}
\newcommand{\rmu}{{\mathrm{u}}}

\newcommand{\cB}{{\mathcal B}}

\newcommand{\cF}{{\mathcal F}}

\newcommand{\cL}{{\mathcal L}}

\newcommand{\cV}{{\mathcal V}}

\newcommand{\ow}{{\overline{w}}}

\DeclareMathOperator{\supp}{supp}

\DeclareMathOperator{\dom}{dom}

\DeclareMathOperator*{\esssup}{ess\,sup}

\renewcommand{\Re}{\text{\rm Re}}

\newcommand{\eps}{\varepsilon}

\newcommand{\beq}{\begin{equation}}
\newcommand{\enq}{\end{equation}}

\let\geq\geqslant
\let\leq\leqslant

\numberwithin{equation}{section}
\allowdisplaybreaks \numberwithin{equation}{section}

\newtheorem{theorem}{Theorem}[section]

\newtheorem{lemma}[theorem]{Lemma}

\newtheorem{definition}[theorem]{Definition}

\theoremstyle{remark}
\newtheorem{remark}[theorem]{Remark}

\newtheorem{example}[theorem]{Example}

\begin{document}

\numberwithin{equation}{section}
\allowdisplaybreaks

\title[Real Spectrum of Second Order Operators on cylinders]{ Spectrum of non-planar traveling waves}

  \author[A.\ Ghazaryan, Y.\ Latushkin, A.\ Pogan]{ Anna Ghazaryan, Yuri Latushkin and Alin Pogan \hspace{80pt}}

%\author{Anna Ghazaryan}
\address{
Department of Mathematics \\
Miami University \\
Oxford, OH  45056,  USA
}
\email{ghazarar@miamioh.edu}
\urladdr{http://www.users.miamioh.edu/ghazarar/}

\address{Department of Mathematics,
University of Missouri, Columbia, MO 65211, USA}
\email{latushkiny@missouri.edu}
\urladdr{http://www.math.missouri.edu/personnel/faculty/latushkiny.html}

\address{Department of Mathematics \\
Miami University \\
Oxford, OH  45056 USA} \email{pogana@miamioh.edu}
\urladdr{http://www.users.miamioh.edu/pogana/}

\thanks{Partially supported by the US National Science
Foundation under Grants NSF DMS-1710989 and NSF DMS-1311313, by the Research Board and Research Council of the University of Missouri, and by the Simons Foundation.}

\date{\today}

\hspace*{-3mm}
%%%%%%%%%%%%%%%%%%%%%%%%%%%%%%%%%%%%%%%%%%

\begin{abstract}
In this paper we prove that a class of non \textit{self-adjoint} second order differential operators acting in cylinders $\Omega\times\RR\subseteq\RR^{d+1}$ have only real discrete spectrum located to the right of the right most point of the essential spectrum. We describe the essential spectrum using the limiting properties of the potential. To track the discrete spectrum we use spatial dynamics and bi-semigroups of linear operators to estimate the decay rate of eigenfunctions associated to isolated eigenvalues.
\end{abstract}

\maketitle

\section{Introduction}\label{s1}

Reaction-diffusion equations are used to model a variety of natural phenomena that occur as a result of interaction of spatial diffusion, convection and reaction of participating variables. In this paper, we consider the reaction-diffusion equation
\begin{equation}\label{RD}
u_t=\Delta_{x,y} u+f(u),\quad (x,y)\in\Omega\times\RR,
\end{equation}
where $f:\RR\to\RR$ is a function of class $C^2$ and the set $\Omega\subseteq\RR^d$ is either a bounded or unbounded domain in $\RR^d$.

In many cases, reaction-diffusion equations exhibit traveling waves, which are special solutions that preserve their shape while moving in a preferred direction.
In systems posed on multi-dimensional domains, such as $\Omega\times\RR$,  a traveling wave is called planar if it is a function  $\widetilde u(t,x,y)=\widetilde u(z)$  of the variable $z= k\cdot (x,y) -ct$, where $k=(k_1,k_2)\in\RR^d\times\RR$ is a constant vector, and  if it is asymptotic to distinct spatially constant steady-state solutions. Without loss of generality, one can take $k=(0,\dots,0,1)$, and, therefore,  $\widetilde u(t,x,y)=\widetilde u(y-ct)$. On the other hand, a solution  of the form $u(t,x,y)=\ou(x,y-ct)$ is called a \textit{non-planar traveling wave}.
The existence of such traveling waves has been established in various cases, using methods such as center-manifold theory, topological methods, maximum principle based arguments, or by exploiting the variational structure of the equation.
Detailed descriptions and specific examples may be found in  \cite{BLL,BN,Berestycki2003,FSV,HaR,LPSS,MoNi,PSS,Rab,Xin} and references therein.

We now briefly describe a non trivial example of $x$-periodic, non-planar traveling waves. This existence result was established in \cite{HaR}.

  %Next, we look for examples of \textit{non-planar} traveling waves solutions of \eqref{RD}.
  %The existence of the following non trivial example of $x$-periodic, non-planar traveling waves was proved in \cite{HaR}.
\begin{example}\label{example-Hamel}
Assume that the function $f$ is of class $C^2$ on an open interval containing $[0,1]$, and there exists $a\in(0,1)$ such that
\begin{equation}\label{hyp-f}
f(0)=f(a)=f(1)=0,\; f(u)\ne 0\;\mbox{for}\;u\in (0,a)\cup(a,1).
\end{equation}
In addition, we assume that $u=0$ and $u=1$ are stable equilibria and $u=a$ is an unstable equilibrium of the diffusion free equation \eqref{RD}, that is,
\begin{equation}\label{bistable}
f'(0)<0,\quad f'(1)<0,\quad f'(a)>0.
\end{equation}
If conditions \eqref{hyp-f} and \eqref{bistable} are satisfied, equation \eqref{RD} is called bistable.
A one-dimensional standing wave of \eqref{RD}, is a solutions of the form $u(t,x,y)=\overline{w}(x)$. The profile $\overline{w}$ satisfies the equation
\begin{equation} \overline{w}_{xx}+f(\overline{w})=0. \label{sch}\end{equation}
Under assumptions \eqref{hyp-f} and \eqref{bistable} this second order equation has homoclinic (spikes), heteroclinic (layers) and periodic solutions.

It is well-known that for any $L>L_{\mathrm{min}}=2\pi\sqrt{f'(a)}$ there exists an unique (up to translations) $L$-periodic solution of equation \eqref{sch} denoted $\ow_L$. In \cite[Theorem 1.1]{HaR} it was proved that \eqref{RD} admits non-planar, $L$-periodic in $x$ solutions connecting the standing wave $\ow_L$ to the equilibrium $u=1$ as $z\to\pm\infty$. More precisely, for each $L>L_{\mathrm{min}}$, there exists a minimal speed $c_L\in\RR$ such that for any $c<c_L$ there exist solutions $\ou_{c,L}(x,z)$ satisfying
\begin{enumerate}
\item[(i)] $\Delta_{x,z}\ou_{c,L}+c\pa_z\ou_{c,L}+f(\ou_{c,L})=0$ in $\RR^2$;
\item[(ii)] $\ou_{c,L}(x,z)\to 1$ as $z\to+\infty$ uniformly in $x\in\RR$;
\item[(iii)] $\ou_{c,L}(x,z)\to \ow_L(x)$ as $z\to-\infty$ uniformly in $x\in\RR$;
\item[(iv)] $\ow_L(x)<\ou_{c,L}(x,z)<1$ for any $x,z\in\RR$;
\item[(v)] $\ou_{c,L}(x+L,z)=\ou_{c,L}(x,z)$ for any $x,z\in\RR$.
\end{enumerate}
\end{example}

Traveling waves (planar and non-planar) are abundant in nature and human activities. In particular, equation \eqref{RD} is a very natural, simple model that describes phenomena arising in chemistry and biology. In this case traveling waves solutions are ubiquitous.  For traveling waves as physical phenomena an important concept is the stability of the waves which describes their resilience under perturbations.
The stability analysis  is based on the information about the location of the spectrum of the operator obtained by linearizing the right hand side of the reaction-diffusion equation  \eqref{RD} about the wave. The spectrum of the linearization consist of  discrete eigenvalues of finite multiplicity and essential spectrum. This   paper  addresses certain properties of the spectrum of this linear operator.

Throughout this paper we assume the existence of a non-planar traveling wave $u(t,x,y)=\ou(x,y-ct)$ of equation \eqref{RD} traveling at speed $c\ne 0$.  In the  variable $z=y-ct$, equation \eqref{RD} becomes
\begin{equation}\label{RD-moving-frame}
u_t=\Delta_{x,z} u+c\pa_z u+f(u).
\end{equation}
We note that $u(t,x,z)=\ou(x,z)$ is a time independent solution of \eqref{RD-moving-frame}. The linearization of \eqref{RD-moving-frame} about $\ou$ is
\begin{equation}\label{RD-linear}
u_t=\cL u, \quad\mbox{where}\quad\cL=\Delta_{x,z}+c\pa_z+V,
\end{equation}
Here $V:\Omega\times\RR\to\RR$ is given by $V(x,z)=f'(\ou(x,z))$. We consider $\cL$ as a closed linear operator on $L^2(\Omega\times\RR)$ with the usual domain $H^2(\Omega\times\RR)\cap H^1_0(\Omega\times\RR)$ when $\Omega\ne\RR^d$ and $H^2(\RR^{d+1})$ when $\Omega=\RR^d$. For simplicity, we write $H^2(\Omega\times\RR)\cap H^1_0(\Omega\times\RR)$ in both cases, slightly abusing the notation.

Our main purpose is to describe the spectrum of the linear operator $\cL$ defined in \eqref{RD-linear}.
The operator $\Delta_{x,z}+V$ is self-adjoint and therefore its spectrum is real. The operator $\cL$ is not self-adjoint and its spectrum is not real. However, we prove that the discrete spectrum of $\cL$ located to the right of the right most point of the essential spectrum is real. This result is  known for the one-dimensional case when there is no $x$-variable, see \cite{BoJo,HLS2}. Its importance stems from the fact that in the typical situation when the essential spectrum is marginally stable, i.e.,  touches the imaginary axis only at 0, the eigenvalues to the right of the right most point of the essential spectrum, if any, give absolute instability of the wave $\ou$, see, e.g., \cite{Sanst}.

We consider \eqref{RD-linear} under the following assumptions:

\noindent{\bf Hypotheses.} \textit{The potential $V:\Omega\times\RR\to\RR$ is bounded. Moreover, there exist two bounded functions $V_\pm:\Omega\to\RR$, such that
\begin{enumerate}
\item[(H1)] $\lim_{z\to\pm\infty}\|V(\cdot,z)-V_\pm(\cdot)\|_{L^\infty(\Omega)}=0$;
\item[(H2)] $z\to \|V(\cdot,z)-V_\pm(\cdot)\|_{L^\infty(\Omega)}$ belongs to $L^1(\RR_\pm)\cap L^\infty(\RR_\pm)$;
\item[(H3)] In the case when the domain $\Omega$ is unbounded we assume that
$$\lim_{r\to\infty}\sup_{x\in\Omega,|x|\geq r}\|V(x,\cdot)-V_\pm(x)\|_{L^\infty(\RR_\pm)}=0.$$
\end{enumerate}}
Hypotheses (H1)-(H3) are satisfied in case of the $x$-periodic, non-planar traveling waves $\ou_{c,L}$ introduced in Example~\ref{example-Hamel}. Indeed,
since $f$ is a smooth function of class $C^2$, using assertions (ii)-(iv), we immediately infer that the potential $V_{c,L}(x,z)=f'(\ou_{c,L}(x,z))$ satisfies Hypothesis (H1) with $\Omega=(0,L)$, $V_{c,L,+}\equiv f'(1)$ and $V_{c,L,-}(x)=f'(\ow_L(x))$. Moreover, using the results from \cite[Remark 1.5, Proposition 4.1]{HaR} (alternatively one can use \cite{BN,LPSS}) we have that the convergence of the non-planar solution $\ou_{c,L}(x,z)$ at $z=\pm\infty$ is exponential in $z$ and uniformly in $x\in (0,L)$, that is there exist $M,\alpha>0$ such that
\begin{equation*}
\|V_{c,L}(\cdot,z)-V_{c,L,\pm}(\cdot)\|_{L^\infty(0,L)}\leq Me^{-\alpha|z|}\quad\mbox{for any}\quad z\in\RR,
\end{equation*}
which proves that the potential $V_{c,L}$ satisfies Hypothesis (H2).

To prove our result we need to use the elementary facts regarding the spectrum of the Schr\"odinger operator $\partial_x^2+V_\pm$. In the case when $\Omega\subseteq\RR^d$ is a bounded domain the spectrum of $\partial_x^2+V_\pm$ consists of a sequence of eigenvalues decreasing to $-\infty$. In the case when $\Omega\subseteq\RR^d$ is unbounded there exist $\alpha_\pm\in\RR$ such that
$\sigma_{\mathrm{ess}}(\partial_x^2+V_\pm)=(-\infty,\alpha_\pm]$ and $\sigma_{\mathrm{d}}(\partial_x^2+V_\pm)$ is a bounded set located to the right of $\alpha_\pm$.

Throughout this paper we denote by $\sigma(T)$ the spectrum, $\rho(T)$ the resolvent set,
$\sigma_\rmd(T)$ the set of isolated eigenvalues of finite algebraic multiplicity of a closed, densely defined linear operator $T:\dom(T)\subseteq\bH\to\bH$ on a Hilbert space $\bH$.
The essential spectrum is defined by $\sigma_{\mathrm{ess}}(T)=\sigma(T)\setminus\sigma_\rmd(T)$ and the essential resolvent set is given by $\rho_{\mathrm{ess}}(T)=\CC\setminus\sigma_{\mathrm{ess}}(T)$.

Following the work of Henry (\cite{He}), first, we describe the essential spectrum of $\cL$ in terms of the limiting operators $$\cL_\pm:H^2(\Omega\times\RR)\cap H^1_0(\Omega\times\RR)\subset L^2(\Omega\times\RR)\to L^2(\Omega\times\RR)$$ defined by
\begin{equation}\label{def-cL-pm}
(\cL_\pm u)(x,z)=\partial_x^2u(x,z)+\partial_z^2u(x,z)+c\partial_zu(x,z)+V_\pm(x)u(x,z).
\end{equation}
Instead of the standard approach, c.f. \cite{He,Sanst} based on exponential dichotomies and Palmer's Theorem \cite{BAG,Pa,Pal}, we use a direct argument based on sequences of approximate eigenfunctions to prove our first major result.
\begin{theorem}\label{t1.1}
Assume Hypotheses (H1)-(H3). Then, the following assertions hold true:
\begin{enumerate}
\item[(i)] $\sigma_{\mathrm{ess}}(\cL_\pm)=\big\{\mu-s^2+cis: s\in\RR,\,\mu\in\sigma(\partial_x^2+V_\pm)\big\}$;
\item[(ii)] $\sigma_{\mathrm{ess}}(\cL_+)\cup \sigma_{\mathrm{ess}}(\cL_-)\subseteq\sigma_{\mathrm{ess}}(\cL)$;
\item[(iii)] $\big\{\lambda\in\CC:\mathrm{Re}\,\lambda>\max\big\{\sup\sigma(\partial_x^2+V_+),\sup\sigma(\partial_x^2+V_-)\big\}\subseteq\rho_{\mathrm{ess}}(\cL)$;
\item[(iv)] $\sup\mathrm{Re}\,\sigma_{\mathrm{ess}}(\cL)=\max\big\{\sup\sigma(\partial_x^2+V_+),\sup\sigma(\partial_x^2+V_-)\big\}$.
\end{enumerate}
\end{theorem}
Our second result shows that adding the term $c\pa_z$ to the self-adjoint operator in \eqref{RD-linear} does not create non-real discrete eigenvalues to the right of the right most point of the essential spectrum. Naturally, we apply the substitution $v(x,z)=e^{cz/2}u(x,z)$ to reduce \eqref{RD-linear} to the self-adjoint case, c.f. \cite[Section 2.3.1]{KapProm}. When $z$ belongs to a bounded interval, obviously, the eigenvalues of $\cL$ in \eqref{RD-linear} are the same as the eigenvalues of the self-adjoint operator $\Delta_{x,z}+V-c^2/4$ and therefore are real. Since in our case $z\in\RR$, we must verify that the potential eigenfunction $v(x,z)=e^{cz/2}u(x,z)$ of the self-adjoint operator belongs to $H^2(\Omega\times\RR)\cap H^1_0(\Omega\times\RR)$. This requires information on the exponential rate of decay of the eigenfunctions of $\cL$. This information is obtained by passing from the second order equation $\cL u=\lambda u$ to the first order system $\pa_zY=A(z)Y$, as usual in spatial dynamics \cite{Kirsh,Mielke1,Mielke2,Sanst,SS-modulated}, see \eqref{eigenfuntion-firstorder}. In the $x$-independent case considered in \cite{HLS2} one can derive the required information on the Lyapunov exponents of this respective asymptotically autonomous system from the asymptotic systems that control the essential spectrum of $\cL$. In the $x$-dependent case considered in the present paper the asymptotically autonomous equation $\pa_zY=A(z)Y$ and the respective asymptotic equations are not well-posed, c.f. \cite{PSS,SS-modulated} and \cite{LP2}. The main technical point of this paper is therefore to overcome this difficulty, and to control the exponential rate of decay of the eigenfunctions of $\cL$ via the spectrum of $\cL_\pm$. This is done in Lemmas~\ref{l3.1}--\ref{l3.7} by means of stable bi-semigroups, c.f. \cite{BGK,LP2}. As a result, we arrive to the following second major result of the paper.

\begin{theorem}\label{t1.2} Assume Hypotheses (H1)-(H3). Then, the discrete spectrum of $\cL$ to the right of the right most point of the essential spectrum is real, that is
\begin{equation}\label{real-point-spectrum}
\big\{\lambda\in\sigma_{\mathrm{point}}(\cL):\mathrm{Re}\,\lambda>\sup\mathrm{Re}\,\sigma_{\mathrm{ess}}(\cL)\big\}\subseteq\RR.
\end{equation}
\end{theorem}
\begin{remark}\label{rem-sys}
We emphasize that using the same techniques, it can be shown that the results of Theorem~\ref{t1.1} and Theorem~\ref{t1.2} hold true in the more general case of systems, where $f:\RR^N\to\RR^N$ provided the matrix-valued potential $V$ is symmetric. The argument requires only minor modifications.
\end{remark}
The difficulty in the proof of Theorem~\ref{t1.2} stems from the fact that the linearization $\cL$ is a multi-dimensional differential operator in $(x,z)\in\Omega\times\RR$. Therefore, we can see the linearization $\cL$ as a second order operator in $z\in\RR$ with operator valued potentials $\pa_x^2+V(\cdot,z)$, which is more general than the case of one-dimensional operators with matrix valued potential. In this case, the eigenvalue problem can be reduced to an infinite-dimensional, first order equation in the non-trivial space $H^1_0(\Omega)\times L^2(\Omega)$ that is not well-posed (see Section~\ref{s3} for details).

The paper is organized as follows: In Section~\ref{s2} we describe the essential spectrum of the linearization $\cL$, proving Theorem~\ref{t1.1}. In Section~\ref{s3} we prove that discrete eigenvalues (if any) to the right of the right most point of the essential spectrum of $\cL$ are real, proving Theorem~\ref{t1.2}. In Section~\ref{s4} we show how one can use the methods described in Section~\ref{s3} to recover the finite dimensional result of \cite{BoJo,HLS2}.

\section{Essential Spectrum}\label{s2}
In this section we describe the essential spectrum of the linear operator $\cL$, proving Theorem~\ref{t1.1}. Throughout this section we assume Hypotheses (H1)-(H3). First, we compute the essential spectrum of the limiting operators $\cL_\pm$. We use perturbation results to describe the connections between the essential spectra of $\cL$ and $\cL_\pm$.

\begin{lemma}\label{r2.1}
Assume Hypotheses (H1)-(H3). Then, the essential spectrum of the linear operator $\cL_\pm$ is given by
\begin{equation}\label{2.1-2}
\sigma_{\mathrm{ess}}(\cL_\pm)=\big\{\mu-s^2+cis: s\in\RR,\,\mu\in\sigma(\partial_x^2+V_\pm)\big\}.
\end{equation}
\end{lemma}
\begin{proof}
We note that taking Fourier Transform in $z\in\RR$, one can readily see that the limiting operator $\cL_\pm$ is similar to  the operator  $M_{\cV_\pm}$ of multiplication by the operator valued function $\cV_\pm(s)=\partial_x^2+V_\pm-s^2+c\rmi s$, $s\in\RR$. For each $s\in\RR$ we consider $\cV_\pm(s)$ as a closed, densely defined linear operator on $L^2(\Omega)$ with domain $H^2(\Omega)\cap H_0^1(\Omega)$.
We infer that
\begin{equation}\label{2.1-1}
\sigma(\cL_\pm)=\sigma(M_{\cV_\pm})=\big\{\mu-s^2+cis: s\in\RR,\,\mu\in\sigma(\partial_x^2+V_\pm)\big\}.
\end{equation}
Since the set above has no isolated points, we infer that the spectrum of $\cL_\pm$ consists entirely of essential spectrum. From \eqref{2.1-1} we conclude that
\begin{equation*}
\sigma_{\mathrm{ess}}(\cL_\pm)=\big\{\mu-s^2+cis: s\in\RR,\,\mu\in\sigma(\partial_x^2+V_\pm)\big\},
\end{equation*}
proving assertion (i) of Theorem~\ref{t1.1}.
\end{proof}

Next, we prove that the essential spectrum of $\cL_\pm$, computed in \eqref{2.1-1}, is contained in the essential spectrum of $\cL$. To prove this result we first show that the essential spectrum of the linear operator $\cL$ is equal to the essential spectrum of the linear operator
$$\cL_\infty: H^2(\Omega\times\RR)\cap H_0^1(\Omega\times\RR)\to L^2(\Omega\times\RR)$$  defined by
\begin{equation}\label{def-cL-infty}
\cL_\infty=\partial_x^2+\partial_z^2+c\partial_z+V_\infty,
\end{equation}
where $V_\infty:\Omega\times\RR\to\RR$ is defined by \begin{equation}\label{2.2-1}
V_\infty(x,z)=\left\{\begin{array}{l l}
	V_+(x) & \; \mbox{if $z\geq0$, }\\
	V_-(x) & \; \mbox{if $z<0$.}\\
	\end{array} \right.
\end{equation}
\begin{lemma}\label{ess-sp-equal}
Assume Hypotheses (H1)-(H3). Then, $\sigma_{\mathrm{ess}}(\cL)=\sigma_{\mathrm{ess}}(\cL_\infty)$.
\end{lemma}
\begin{proof}
Since $\cL=\cL_\infty+M_{V-V_\infty}$, where $M_{V-V_\infty}$ is the operator of multiplication on $L^2(\Omega\times\RR)$ by the bounded function $V-V_\infty$, to prove the lemma it is enough to
show that $M_{V-V_\infty}$ is relatively compact with respect to $\cL$. First, we approximate the operator $M_{V-V_\infty}$ by an operator of multiplication by a function with compact support.

Let $\{\rho_n\}_{n\geq 1}$ be a sequence of $C^\infty$ functions satisfying the conditions: $0\leq\rho_n\leq 1$, $\rho_n(z)=1$ for any $z\in [-n,n]$,
$\rho_n(z)=0$ if $|z|\geq n+1$. We define the sequence of functions $K_n:\Omega\times\RR\to\RR$ by
\begin{equation}\label{def-Kn}
K_n(x,z)=\rho_n(|x|)\rho_n(z)\big(V(x,z)-V_\infty(x,z)\big).
\end{equation}
Since $\supp K_n\subseteq\widetilde{\Omega}_n:=(\Omega\times\RR)\cap(-n-1,n+1)^{d+1}$ and $H^2(\widetilde{\Omega}_n)$ is compactly embedded in $L^2(\widetilde{\Omega}_n)$ for any $n\geq 1$, we infer that $M_{K_n}$, the linear operator on $L^2(\Omega\times\RR)$ of multiplication by $K_n$, is relatively compact with respect to $\cL$ for any $n\geq 1$. To finish the proof of lemma, it is enough to show that $M_{K_n}\to M_{V-V_\infty}$ as $n\to\infty$ in the operator norm.
From \eqref{2.2-1}, Hypothesis (H1) and Hypothesis (H3) it follows that
\begin{align}\label{V-Vinfty}
&\lim_{z\to\pm\infty}\|V(\cdot,z)-V_\infty(\cdot,z)\|_{L^\infty(\Omega)}=0,\nonumber\\
&\lim_{r\to\infty}\sup_{x\in\Omega,|x|\geq r}\|V(x,\cdot)-V_\infty(x,\cdot)\|_{L^\infty(\RR)}=0.
\end{align}
Using the definition of $K_n$ in \eqref{def-Kn} we obtain that
\begin{align}\label{m-est}
\|M_{K_n}&- M_{V-V_\infty}\|_{L^2\to L^2}=\|K_n-(V-V_\infty)\|_{L^\infty(\Omega\times\RR)}\nonumber\\
&=\esssup_{(x,z)\in\Omega\times\RR}\big(1-\rho_n(|x|)\rho_n(z)\big)|V(x,z)-V_\infty(x,z)| \nonumber\\
&\leq \esssup_{(x,z)\in\Omega\times\RR}\big(2-\rho_n(|x|)-\rho_n(z)\big)|V(x,z)-V_\infty(x,z)|\nonumber\\
&\leq\sup_{x\in\Omega}\big(|1-\rho_n(|x|)|\,\|V(x,\cdot)-V_\infty(x,\cdot)\|_{L^\infty(\RR)}\big)\nonumber\\
&\qquad+\sup_{z\in\RR}\big(|1-\rho_n(z)|\,\|V(\cdot,z)-V_\infty(\cdot,z)\|_{L^\infty(\Omega)}\big)\nonumber\\
&\leq\sup_{x\in\Omega,|x|\geq n}\big(\|V(x,\cdot)-V_\infty(x,\cdot)\|_{L^\infty(\RR)}\big)\nonumber\\
&\qquad+\sup_{|z|\geq n}\big(\|V(\cdot,z)-V_\infty(\cdot,z)\|_{L^\infty(\Omega)}\big).
\end{align}
From \eqref{V-Vinfty} and \eqref{m-est} we conclude that $\|M_{K_n}- M_{V-V_\infty}\|_{L^2\to L^2}\to 0$ as $n\to\infty$, proving the lemma.
\end{proof}
\begin{lemma}\label{l2.2}
Assume Hypotheses (H1)-(H3). Then, $\sigma_{\mathrm{ess}}(\cL_+)\cup \sigma_{\mathrm{ess}}(\cL_-)\subseteq\sigma_{\mathrm{ess}}(\cL)$.
\end{lemma}
\begin{proof}
First, we prove that $\sigma_{\mathrm{ess}}(\cL_+)\subseteq\sigma_{\mathrm{ess}}(\cL_\infty)$ using sequences of approximate eigenfunctions. Similarly, we can show that $\sigma_{\mathrm{ess}}(\cL_-)\subseteq\sigma_{\mathrm{ess}}(\cL_\infty)$. From Lemma~\ref{r2.1}, for a fixed $\lambda\in\sigma_{\mathrm{ess}}(\cL_+)$ we have that there exist $a\in\RR$ and $\mu\in\sigma(\partial_x^2+V_+)$ such that $\lambda=\mu-a^2+c\rmi a$. Since $\partial_x^2+V_+$ is a self-adjoint linear operator on $L^2(\Omega)$, we obtain that there exists a sequence  $\{\varphi_n\}_{n\geq 1}$  in $H^2(\Omega)\cap H_0^1(\Omega)$ (\cite[Chapter IX]{EdEv}) such that $\|\varphi_n\|_2=1$ for any $n\geq 1$ and
\begin{equation}\label{2.2-2}
f_n:=\varphi_n''+\big(V_+(\cdot)-\mu\big)\varphi_n\to 0\quad\mbox{in}\quad L^2(\Omega)\quad\mbox{as}\quad n\to\infty.
\end{equation}
We then take  $\psi\in C_0^\infty(\RR)$, $\psi\ne 0$, such that $\supp\psi\subseteq [0,1]$, and  define $u_n:\RR\to\CC$ by $u_n(z)=n^{-1/2}e^{\rmi az}\psi(n^{-1}z-n)$. Since $\supp\psi\subseteq [0,1]$ it follows that
\begin{equation}\label{2.2-3}
\supp u_n\subseteq [n^2,n^2+n]\quad\mbox{for any}\quad n\geq 1.
\end{equation}
Moreover, one can readily check that
\begin{equation}\label{2.2-4}
\|u_n\|_2^2=\frac{1}{n}\int_\RR |\psi(n^{-1}z-n)|^2\rmd z=\int_\RR |\psi(y)|^2\rmd y=\|\psi\|_2^2> 0
\end{equation}
and
\begin{equation}u_n''+cu_n'+(a^2-c\rmi a)u_n=(c+2\rmi a)v_n+w_n \text{ for any } n\geq 1, \label{275} \end{equation}  where $v_n,w_n:\RR\to\CC$ are defined by
\begin{equation}\label{2.2-5}
v_n(z)=\frac{e^{\rmi az}}{\sqrt{n^3}}\psi'(n^{-1}z-n),\quad w_n(z)=\frac{e^{\rmi az}}{\sqrt{n^5}}\psi'(n^{-1}z-n).
\end{equation}
Integrating, we obtain that
\begin{equation}\label{2.2-6}
\|v_n\|_2=\frac{1}{n}\|\psi'\|_2\quad\mbox{and}\quad \|w_n\|_2=\frac{1}{n^2}\|\psi''\|_2\quad\mbox{for any}\quad n\geq 1.
\end{equation}
We note that assertions \eqref{2.2-4}, \eqref{2.2-5}, and \eqref{2.2-6} show that $\{u_n\}_{n\geq 1}$ is a sequence of approximate eigenvalues of the constant coefficient operator $\partial_z^2+c\partial_z$. From \eqref{2.2-3} we have that $\supp u_n\subseteq\RR_+$, therefore
\begin{equation}\label{2.2-7}
u_n(z)V_\infty(x,z)=u_n(z)V_+(x)\quad\mbox{for any}\quad n\geq 1,\; x\in\Omega,\;z\in\RR.
\end{equation}
We introduce $\phi_n:\Omega\times\RR\to\CC$ by $\phi_n(x,z)=\varphi_n(x)u_n(z)$.
From  \eqref{275} and \eqref{2.2-7} it follows that for any $n\geq 1$, $x\in\Omega$ and $z\in\RR$,
\begin{align}\label{2.2-8}
\big((\cL_\infty-\lambda)\phi_n\big)(x,z)&=\varphi_n''(x)u_n(z)+\varphi_n(x)u_n''(z)+c\varphi_n(x)u_n'(z)\nonumber\\
&\quad+V_\infty(x,z)\varphi_n(x)u_n(z)-\lambda\varphi_n(x)u_n(z)\nonumber\\
&=\varphi_n(x)\big(u_n''(z)+cu_n'(z)+(a^2-c\rmi a)u_n(z)\big)\nonumber\\
&\quad+u_n(z)\big(\varphi_n''(x)+\big(V_+(x)-\mu\big)\varphi_n(x)\big)\nonumber\\
&=f_n(x)u_n(z)+\big((c+2\rmi a)v_n(z)+w_n(z)\big)\varphi_n(x).
\end{align}
 Since $\|\varphi_n\|_2=1$ for any $n\geq 1$ from \eqref{2.2-4} we infer that
\begin{equation}\label{2.2-9}
\|\phi_n\|_2=\|\varphi_n\|_2\|u_n\|_2=\|\psi\|_2>0\quad\mbox{for any}\quad n\geq 1.
\end{equation}
In addition, from \eqref{2.2-2} and \eqref{2.2-6} we have that $f_n\to 0$ in $L^2(\Omega)$, $v_n\to 0$ and $w_n\to 0$ in $L^2(\RR)$ as $n\to\infty$. From \eqref{2.2-8} we conclude that
\begin{equation}\label{2.2-10}
(\cL_\infty-\lambda)\phi_n\to 0\quad\mbox{in}\quad L^2(\Omega\times\RR)\quad\mbox{as}\quad n\to\infty.
\end{equation}
From \eqref{2.2-9} and \eqref{2.2-10} it follows that $\lambda\in\sigma(\cL_\infty)$. Summarizing, we have that
\begin{equation}\label{2.2-11}
\sigma_{\mathrm{ess}}(\cL_+)=\big\{\mu-s^2+cis: s\in\RR,\,\mu\in\sigma(\partial_x^2+V_+)\big\}\subseteq\sigma(\cL_\infty).
\end{equation}
Since $\sigma_{\mathrm{ess}}(\cL_+)$ does not have isolated points and $\sigma_{\mathrm{ess}}(\cL_\infty)=\sigma_{\mathrm{ess}}(\cL)$, we conclude that $\sigma_{\mathrm{ess}}(\cL_+)\subseteq\sigma_{\mathrm{ess}}(\cL)$, thus  proving the lemma.
\end{proof}
Next, we prove that the resolvent set of the linear operator $\cL_\infty$ defined in \eqref{def-cL-infty} contains a right half-plane.

\begin{lemma}\label{l2.3}
Assume Hypotheses (H1)-(H3). Then, the following inclusion holds true:
$$\big\{\lambda\in\CC:\mathrm{Re}\,\lambda>\max\{\sup\sigma(\partial_x^2+V_+),\sup\sigma(\partial_x^2+V_-)\}\big\}\subseteq\rho(\cL_\infty).$$
\end{lemma}

\begin{proof} Fix $\lambda_0\in\CC$ such that $\mathrm{Re}\,\lambda_0>\sup\sigma(\partial_x^2+V_\pm)$. Since the linear operators
$\partial_x^2+V_\pm$ are self-adjoint, from the min-max principle (see, e.g. \cite{RD4}), we have that there exists $\eps_0>0$ such that $\partial_x^2+V_\pm\leq(\mathrm{Re}\,\lambda_0-\eps_0)I$, that is
\begin{equation}\label{2.3-1}
\big\langle (\partial_x^2+V_\pm-\mathrm{Re}\,\lambda_0)g,g\big\rangle_{L^2(\Omega)}\leq -\eps_0\|g\|^2_{2}\quad  \mbox{ for any }\, g\in H^2(\Omega)\cap H_0^1(\Omega).
\end{equation}
We define the operator valued function
\begin{equation}\label{2.3-2}
F_\infty(z)=\left\{\begin{array}{l l}
	\partial_x^2+V_+ & \; \mbox{if $z\geq0$, }\\
	\partial_x^2+V_- & \; \mbox{if $z<0$.}\\
	\end{array} \right. %,\quad z\in\RR.
\end{equation}
For each $z\in\RR$, we consider $F_\infty(z)$ as a closed, densely defined linear operator on $L^2(\Omega)$ with domain $H^2(\Omega)\cap H_0^1(\Omega)$.
The inequality \eqref{2.3-1} is equivalent to
\begin{equation}\label{2.3-3}
\big\langle (F_\infty(z)-\mathrm{Re}\,\lambda_0)g,g\big\rangle_{L^2(\Omega)}\leq -\eps_0\|g\|^2_{2}\;\mbox{for any}\; z\in\RR,\;g\in H^2(\Omega)\cap H_0^1(\Omega).
\end{equation}
The linear operator $\cL_\infty$ admits the representation
\begin{equation}\label{2.3-4}
\cL_\infty=\partial_z^2+c\partial_z+F_\infty.
\end{equation}
Since $F_\infty(z)$ is self-adjoint for any $z\in\RR$ it follows that $\dom(\cL_\infty^*)=\dom(\cL_\infty)$.
Moreover, from \eqref{2.3-3} we obtain that
\begin{align}\label{2.3-5}
\mathrm{Re}\,\big\langle (\lambda_0-\cL_\infty)v,v\big\rangle_{L^2(\Omega\times\RR)}&=\langle (\mathrm{Re}\,\lambda_0-\mathrm{Re}\,\cL_\infty)v,v\rangle_{L^2(\Omega\times\RR)}\nonumber\\
&=\mathrm{Re}\,\lambda_0\|v\|_{2}^2+\|\partial_zv\|_{2}^2-\langle F_\infty(\cdot)v,v\rangle_{L^2(\Omega\times\RR)}\nonumber\\
&\geq\langle (\mathrm{Re}\,\lambda_0-F_\infty(\cdot))v,v\rangle_{L^2(\Omega\times\RR)}\nonumber\\
&=\int_\RR \langle (\mathrm{Re}\,\lambda_0-F_\infty(z))v(\cdot,z),v(\cdot,z))\rangle_{L^2(\Omega)}\rmd z\nonumber\\
&\geq\eps_0\int_\RR \langle v(\cdot,z),v(\cdot,z))\rangle_{L^2(\Omega)}\rmd z=\eps_0\|v\|_{2}^2
\end{align}
for any $v\in\dom(\cL_\infty)=\dom(\cL_\infty^*)$. From \eqref{2.3-5} we conclude that $\ker(\overline{\lambda}_0-\cL_\infty^*)=\{0\}$ and
\begin{equation}\label{2.3-6}
\|(\lambda_0-\cL_\infty)v\|_{L^2(\Omega\times\RR)}\geq\eps_0\|v\|_{L^2(\Omega\times\RR)}\quad\mbox{for any}\quad v\in\dom(\cL_\infty),
\end{equation}
which implies that $\lambda_0-\cL_\infty$ is invertible, proving the lemma.
\end{proof}

\noindent{\textbf{Proof of Theorem~\ref{t1.1}.}} We have now all the ingredients needed to describe the essential spectrum of $\cL$. Assertions (i) and (ii) were proved in Lemma~\ref{r2.1} and Lemma~\ref{l2.2}, respectively. Since $\sigma_{\mathrm{ess}}(\cL_\infty)=\sigma_{\mathrm{ess}}(\cL)$ by Lemma~\ref{ess-sp-equal}, assertion (iii) follows shortly from Lemma~\ref{l2.3}. Assertion (iv) is a direct consequence of (ii) and (iii).

\section{Bi-semigroups and eigenfunction decay rate}\label{s3}

In this section we prove that the eigenvalues to the right of the right most point of the  essential spectrum of $\cL$ must be real. The main idea of the proof  is to make a change of variables in the eigenvalue equation and then prove that the eigenvalues of $\cL$ to the right of the right most point of the essential spectrum are eigenvalues of a second order, self-adjoint linear operator. To achieve this goal we need to estimate the decay rate of eigenfunctions by reducing the second order eigenvalue equation to a first order linear differential equation on a suitable Hilbert space, and then use perturbation results to prove Theorem~\ref{t1.2}.

We assume $\lambda_0\in\sigma_{\rmd}(\cL)$ is such that $\mathrm{Re}\,\lambda_0>\sup\mathrm{Re}\,\sigma_{\mathrm{ess}}(\cL)$. Therefore, there exists an eigenfunction $u_0\in H^2(\Omega\times\RR)\cap H_0^1(\Omega\times\RR)$ satisfying
\begin{equation}\label{eigenfuntion-u0}
\partial_x^2u_0+\partial_z^2u_0+c\partial_zu_0+V(x,z)u_0=\lambda_0u_0.
\end{equation}
Next, we introduce the function $v_0:\Omega\times\RR\to\CC$ by $v_0(x,z)=e^{cz/2}u_0(x,z)$. One can readily check that $v_0\in H^2_{\mathrm{loc}}(\Omega\times\RR)$ and
\begin{equation}\label{eigenfuntion-v0}
\partial_x^2v_0+\partial_z^2v_0+\big(V(x,z)-\frac{c^2}{4}\big)v_0=\lambda_0v_0.
\end{equation}
To prove our result we need to show that $v_0$ is a genuine eigenfunction, that it belongs to $H^2(\Omega\times\RR)\cap H_0^1(\Omega\times\RR)$. To do this, we  study its decay rate using the spatial dynamics method by treating $z\in\RR$ as time in \eqref{eigenfuntion-v0}. First, we introduce $w_0=\partial_zv_0\in H^1_{\mathrm{loc}}(\Omega\times\RR)$. The pair of functions $(v_0,w_0)^{\mathrm{T}}$ satisfies the first order differential equation
\begin{equation}\label{eigenfuntion-firstorder}
\partial_z\begin{pmatrix}v_0\\w_0\end{pmatrix}=A(z)\begin{pmatrix}v_0\\w_0\end{pmatrix},
\end{equation}
where for each $z\in\RR$ the linear operator $A(z)$ acting in the space $H^1_0(\Omega)\times L^2(\Omega)$ is defined by
\begin{equation}\label{def-A}
A(z)=\begin{bmatrix}
0 & I \\
\lambda_0+\frac{c^2}{4}-V(\cdot,z)-\partial_x^2 & 0
\end{bmatrix}
\end{equation}
with domain $\dom(A(z))=\big(H^2(\Omega)\cap H_0^1(\Omega)\big)\times H^1_0(\Omega)$.
We note that for each $z\in\RR$ the linear operator $A(z)$ can be obtained as the bounded perturbation of the linear operator acting in the space $H^1_0(\Omega)\times L^2(\Omega)$ and defined by
\begin{equation}\label{def-Apm}
A_\pm=\begin{bmatrix}
0 & I \\
\lambda_0+\frac{c^2}{4}-V_\pm(\cdot)-\partial_x^2 & 0
\end{bmatrix}.
\end{equation}
with domain $\dom(A_\pm)=\big(H^2(\Omega)\cap H_0^1(\Omega)\big)\times H^1_0(\Omega)$. The spectrum of the linear operator $A_\pm$ can be computed in terms of the spectrum of the self-adjoint operator $\partial_x^2+V_\pm$ as follows:
\begin{equation}\label{square-root-spectrum}
\sigma(A_\pm)=\Big\{\pm\sqrt{\lambda_0+c^2/4-\mu}:\mu\in\sigma(\partial_x^2+V_\pm)\Big\}.
\end{equation}
Hence, the real part of the spectrum of $A_\pm$ is unbounded from below and from above, therefore it cannot generate a $C^0$-semigroup.
We conclude that equation \eqref{eigenfuntion-firstorder} is not well-posed.
\begin{remark}\label{BackForward-uniqueness} In the case of the not well-posed first order differential equation \eqref{eigenfuntion-firstorder} one cannot immediately infer the existence of backward nor forward propagators.
Hence, when studying exponential dichotomy of \eqref{eigenfuntion-firstorder}, in \cite{PSS,LP2} the following backward-forward uniqueness was assumed:
\begin{itemize}
\item[(i)] If $Y$ is a solution of \eqref{eigenfuntion-firstorder} on $\RR$ and $Y(z_0)=0$ for some $z_0\in\RR$, then $Y\equiv0$;
\item[(ii)] If $Z$ is a solution of the adjoint equation $Z'=-A(z)^*Z$
on $\RR$ and $Z(z_0)=0$ for some $z_0\in\RR$, then $Z\equiv0$.
\end{itemize}
This assumption was used to prove exponential dichotomies, more precisely, to prove the existence and uniqueness of solutions on the semi-lines $(-\infty,a]$ and $[b,\infty)$ that decay exponentially in backward and forward time, respectively.
In our case we do not need to assume this hypothesis since the only solution of \eqref{eigenfuntion-firstorder} we are particularly interested in, $(w_0,\pa_zv_0)^{\mathrm{T}}$, exists because $\lambda_0$ is a discrete eigenvalue of $\cL$, therefore we have a non-zero solution of
equations \eqref{eigenfuntion-u0} and \eqref{eigenfuntion-v0}.
\end{remark}
Our first task is to prove that $A_\pm$ generates a bi-semigroup on $H_0^1(\Omega)\times L^2(\Omega)$. Here we recall that a closed, densely defined linear operator $G$ generates a bi-semigroup on a Hilbert space $\bH$, if there exist $\bH_\rms$ and $\bH_\rmu$ two closed subspaces of $\bH$ invariant under $G$ such that $\bH=\bH_\rms\oplus\bH_\rmu$ (here $\oplus$ is a direct sum, not necessarily orthogonal) and $G_{|\bH_\rms}$ and $-G_{|\bH_\rmu}$ generate $C^0$ semigroups on $\bH_\rms$ and $\bH_\rmu$, denoted by $\{T_\rms(z)\}_{z\geq 0}$ and $\{T_\rmu(z)\}_{z\geq 0}$, respectively. We say that the bi-semigroup has decay rate $-\nu<0$ if there exists $C>0$ such that
\begin{equation}\label{stabe-bi-sem}
\|T_\rms(z)\|\leq Ce^{-\nu z}\quad\mbox{and}\quad\|T_\rmu(z)\|\leq Ce^{-\nu z}\quad\mbox{for any}\quad z\geq 0.
\end{equation}
To prove that $A_\pm$ generates a bi-semigroup on $\bH=H^1_0(\Omega)\times L^2(\Omega)$ we use the following abstract result.
\begin{lemma}\label{l3.1}
Assume $\bH$ is a Hilbert space, $\beta\in\RR$ and $T:\dom(T)\subset\bH\to\bH$ is a closed, densely defined, self-adjoint linear operator on $\bH$ satisfying the condition
\begin{equation}\label{spectrum-T}
\sigma(T)\subseteq [\alpha,\infty)\quad\mbox{for some}\quad \alpha>0.
\end{equation}
If $\bY=\dom(|T|^{1/2})\times\bH$, then the linear operator $G:\dom(G)\subset\bY\to\bY$ defined by
\begin{equation}\label{def-G}
\dom(G)=\dom(T)\times\dom(|T|^{1/2}),\quad G=\begin{bmatrix}
0 & I \\
\rmi\beta+T & 0
\end{bmatrix}
\end{equation}
generates an exponentially stable bi-semigroup on $\bY$ having decay rate $-\nu$, satisfying the condition $\nu\geq\sqrt{\alpha}$.
\end{lemma}
\begin{proof}
Since $T$ is a closed, densely defined, positive self-adjoint linear operator on $\bH$, by the Spectral Theorem we have that there exists $(\Gamma,\mu)$,  a measure space, $\gamma:\Gamma\to\RR$, a $\mu$-measurable function, and $U:\bH\to L^2(\Gamma,\mu)$, a unitary operator, such that
\begin{equation}\label{3.1-1}
T=U^*M_{\gamma^2}U.
\end{equation}
Here $M_{\gamma^2}$ denotes the operator of multiplication on $L^2(\Gamma,\mu)$ by the function $\gamma^2$. Using this representation one can readily check that
\begin{align}\label{3.1-2}
\dom(|T|)&=\{h\in\bH:\gamma^2Uh\in L^2(\Gamma,\mu)\},\nonumber\\
\dom(|T|^{1/2})&=\{h\in\bH:|\gamma|Uh\in L^2(\Gamma,\mu)\}.
\end{align}
From \eqref{spectrum-T} and \eqref{3.1-1} we obtain that $\sigma(M_{\gamma^2})=\sigma(T)=[\alpha,\infty)$. Since the spectrum of multiplication operators is given by their essential range, we infer that
\begin{equation}\label{3.1-3}
|\gamma(\omega)|\geq \sqrt{\alpha}\;\mbox{for}\; \mu-\mbox{almost all}\;\;\omega\in\Gamma.
\end{equation}
Let $S:\dom(M_{|\gamma|})\to L^2(\Gamma,\mu)$ be the multiplication operator by the complex valued function $\sqrt{\gamma^2+\rmi\beta}$. Here $\sqrt{\lambda}$ denotes the principal branch of the complex square root of $\lambda\in\CC$. Using \eqref{3.1-1} we have that
\begin{equation}\label{3.1-4}
T+\rmi\beta=U^*S^2U.
\end{equation}
From \eqref{3.1-3} it follows that
\begin{equation}\label{3.1-5}
\mathrm{Re}\,\sqrt{\gamma^2(\omega)+\rmi\beta}=\sqrt{\frac{\gamma^2(\omega)+\sqrt{\gamma^4(\omega)+\beta^2}}{2}}\geq |\gamma(\omega)|\geq\sqrt{\alpha}
\end{equation}
for $\mu$-almost all $\omega\in\Gamma$, which implies that $S=M_{\sqrt{\gamma^2+\rmi\beta}}$ is invertible with bounded inverse. We define $W:\dom(M_{|\gamma|})\times L^2(\Gamma,\mu)\to L^2(\Gamma,\mu)\times L^2(\Gamma,\mu)$ by
\begin{equation}\label{3.1-6}
W=\frac{1}{\sqrt{2}}\begin{bmatrix}
S & -I \\
S & I
\end{bmatrix}.
\end{equation}
Since $S$ is invertible with bounded inverse, we infer that $W$ is invertible with bounded inverse and
\begin{equation}\label{3.1-7}
W^{-1}=\frac{1}{\sqrt{2}}\begin{bmatrix}
S^{-1} & S^{-1} \\
-I & I
\end{bmatrix}.
\end{equation}
Let $\widetilde{G}:\dom(M_{|\gamma|})\times \dom(M_{|\gamma|})\to L^2(\Gamma,\mu)\times L^2(\Gamma,\mu)$ be the linear operator defined by
\begin{equation}\label{3.1-8}
\widetilde{G}=\frac{1}{\sqrt{2}}\begin{bmatrix}
-S & 0 \\
0 & S
\end{bmatrix}.
\end{equation}
From \eqref{3.1-5} we obtain that the linear operator $-S=M_{-\sqrt{\gamma^2+\rmi\beta}}$ generates a $C^0$-semigroup having decay rate $-\nu$, for some $\nu\geq\sqrt{\alpha}$. Hence, $\widetilde{G}$ generates a stable bi-semigroup having decay rate $-\nu$. From \eqref{3.1-4} and \eqref{3.1-6}--\eqref{3.1-8} we conclude that
\begin{equation}\label{3.1-9}
G=\begin{bmatrix}
0 & I \\
\rmi\beta+T & 0
\end{bmatrix}=\begin{bmatrix}
U^{-1} & 0 \\
0& U^{-1}
\end{bmatrix}
W^{-1}\widetilde{G}W
\begin{bmatrix}
U & 0 \\
0& U
\end{bmatrix},
\end{equation}
proving the lemma.
\end{proof}
\begin{lemma}\label{l3.2}
Assume Hypotheses (H1)-(H3) and $\Re\;\lambda_0>\sup\Re\,\sigma_{\mathrm{ess}}(\cL)$. Then, the linear operator $A_\pm$ defined in \eqref{def-Apm} generates an exponentially stable bi-semigroup on the Hilbert space $H_0^1(\Omega)\times L^2(\Omega)$ having decay rate $-\nu_\pm$ satisfying the condition $\nu_\pm>|c|/2$.
\end{lemma}
\begin{proof}
The lemma follows from Lemma~\ref{l3.1} above and Theorem~\ref{t1.1}(iv), describing the connection between $\sigma_{\mathrm{ess}}(\cL)$ and $\sigma(\pa_x^2+V_\pm)$. Indeed,
\begin{equation}\label{3.2-1}
T_\pm=\mathrm{Re}\,\lambda_0+\frac{c^2}{4}-\pa_x^2-V_\pm(x)
\end{equation}
is a closed, densely defined, self-adjoint linear operator on $L^2(\Omega)$ having domain $\dom(T_\pm)=H^2(\Omega)\cap H_0^1(\Omega)$. Moreover, since $\mathrm{Re}\,\lambda_0>\sup\mathrm{Re}\,\sigma_{\mathrm{ess}}(\cL)=\max\{\max\sigma(\partial_x^2+V_+),\max\sigma(\partial_x^2+V_-)\}$, we infer that
\begin{equation}\label{3.2-2}
\sigma(T_\pm)\subseteq [\alpha_*,\infty),\quad\mbox{where}\quad\alpha_*=\mathrm{Re}\,\lambda_0-\sup\mathrm{Re}\,\sigma_{\mathrm{ess}}(\cL)+\frac{c^2}{4}.
\end{equation}
In addition, one can readily check that $\dom(|T_\pm|^{1/2})=H_0^1(\Omega)$. Applying Lemma~\ref{l3.1}, we conclude that
\begin{equation}\label{3.2-3}
A_\pm=\begin{bmatrix}
0 & I \\
\rmi\mathrm{Im}\,\lambda_0+T_\pm & 0
\end{bmatrix}
\end{equation}
generates an exponentially stable bi-semigroup on $H_0^1(\Omega)\times L^2(\Omega)$ having decay rate $-\nu_\pm$. From \eqref{3.2-2} and since $\mathrm{Re}\,\lambda_0>\sup\mathrm{Re}\,\sigma_{\mathrm{ess}}(\cL)$ it follows that $\nu_\pm\geq\sqrt{\alpha_*}>|c|/2$.
\end{proof}
We turn our attention to equation $Y'=A(z)Y$. We note that it is equivalent to
\begin{equation}\label{ill-possed}
Y'=\big(A_\pm+B_\pm(z)\big)Y,
\end{equation}
where the operator valued function $B_\pm:\RR\to\cB(H_0^1(\Omega)\times L^2(\Omega))$ is defined by
\begin{equation}\label{def-Bpm}
B_\pm(z)=\begin{bmatrix}
0 & 0 \\
V_\pm(\cdot)-V(\cdot,z)&0\end{bmatrix}.
\end{equation}
\begin{lemma}\label{r3.3}
Assume Hypotheses (H1)-(H3). Then, the following assertions hold true:
\begin{enumerate}
\item[(i)] $B_\pm$ is strongly continuous;
\item[(ii)] $\|B_\pm(\cdot)\|_{\cB(H_0^1(\Omega)\times L^2(\Omega))}\in L^1(\RR_\pm)\cap L^\infty(\RR_\pm)$.
\end{enumerate}
\end{lemma}
\begin{proof} Let $f\in H_0^1(\Omega)$ and $g\in L^2(\Omega)$. From \eqref{def-Bpm} we have that
\begin{equation*}
 \big(B_\pm(z_1)-B_\pm(z_2)\big)(f,g)^{\mathrm{T}}=(0,V(\cdot,z_1)f-V(\cdot,z_2)f)^{\mathrm{T}}\;\mbox{for any}\; z_1,\,z_2\in\RR.
\end{equation*}
Assertion (i) follows from Lebesgue Dominated Convergence Theorem and the fact that $V(\cdot,\cdot)$ is bounded and continuous on $\Omega\times\RR$. Moreover,
\begin{align*}
 \big\|B_\pm(z)(f,g)^{\mathrm{T}}\big\|_{H_0^1(\Omega)\times L^2(\Omega}^2&=\int_\RR |V(x,z)-V_\pm(x)|^2|f(x)|^2\rmd x\\
&\leq \|f\|_{L^2(\Omega)}^2\|V(\cdot,z)-V_\pm(\cdot)\|_{L^\infty(\Omega)}^2\\
&\leq \|f\|_{H^1(\Omega)}^2\|V(\cdot,z)-V_\pm(\cdot)\|_{L^\infty(\Omega)}^2
\end{align*}
which implies that $\|B_\pm(z)\|\leq \|V(\cdot,z)-V_\pm(\cdot)\|_{L^\infty(\Omega)}$ for any $z\in\RR$. From Hypothesis (H2), we immediately infer that $\|B_\pm(\cdot)\|\in L^1(\RR_\pm)\cap L^\infty(\RR_\pm)$.
\end{proof}
In what follows, we denote by $\{T_{\rms/\rmu}^\pm(z)\}_{z\geq 0}$ the semigroups generated by $A_\pm$ on its invariant stable/unstable subspaces $\bH_{\rms/\rmu}^\pm$. Furthermore, we denote by $P_{\rms/\rmu}^\pm$ the projection into $\bH_{\rms/\rmu}^\pm$ parallel to $\bH_{\rmu/\rms}^\pm$. We recall from \cite{PSS} the definition of mild solutions of the not well-posed equation \eqref{ill-possed}.
\begin{definition}\label{d3.4}
We say that a continuous function $Y:[a,b]\to H_0^1(\Omega)\times L^2(\Omega)$ is a \textit{mild solution} of \eqref{ill-possed} on $[a,b]$ if the following equation holds:
\begin{align}\label{mild}
Y(z)&=T_\rms^\pm(z-a)P_\rms^\pm Y(a)+\int_a^z T_\rms^\pm(z-\zeta)P_\rms^\pm B_\pm(\zeta)Y(\zeta)\rmd\zeta\nonumber\\
&+T_\rmu^\pm(b-z)P_\rmu^\pm Y(b)-\int_z^b T_\rmu^\pm(\zeta-z)P_\rmu^\pm B_\pm(\zeta)Y(\zeta)\rmd\zeta
\end{align}
for any $z\in [a,b]$.
\end{definition}
\begin{remark}\label{r3.5}
Using a frequency domain reformulation, it was shown in \cite{LP2} that \eqref{mild} is equivalent to
\begin{equation}\label{mild-F}
\big((\cF-M_{R_\pm}\cF M_{B_\pm})Y_{|[a,b]}\big)(\xi)=R_\pm(\xi)(e^{-2\pi\rmi\xi a}Y(a)-e^{-2\pi\rmi\xi b}Y(b))
\end{equation}
for any $\xi\in\RR$, where $R_\pm:\RR\to\cB(H_0^1(\Omega)\times L^2(\Omega))$ is defined by $R_\pm(\xi)=(2\pi\rmi\xi-A_\pm)^{-1}$, $M_{R_\pm}$ and $M_{B_\pm}$ denote the operators of multiplication on $L^2(\RR,H_0^1(\Omega)\times L^2(\Omega))$ by the operator valued functions $R_\pm$ and $B_\pm$, respectively.
\end{remark}
\begin{remark}\label{r3.6}
Going back to the eigenvalue problems \eqref{eigenfuntion-u0}--\eqref{eigenfuntion-v0}, since $u_0\in H^2(\Omega\times\RR)\cap H^1_0(\Omega\times\RR)$ and $v_0(x,z)=e^{cz/2}u_0(x,z)$, we can immediately infer that the function $Y_0:\RR\to H_0^1(\Omega)\times L^2(\Omega)$ defined by
\begin{equation}\label{def-Y0}
Y_0(z)=(v_0(\cdot,z),\pa_zv_0(\cdot,z))^{\mathrm{T}}
\end{equation}
is continuous. Taking Fourier transform in \eqref{eigenfuntion-v0}, we infer that $Y_0$ satisfies equation
\eqref{mild-F} for any $a<b$. From Remark~\ref{r3.5} conclude that $Y_0$ is a mild solution of equation $Y'=\big(A_\pm+B_\pm(z)\big)Y$ on $[a,b]$ for any $a<b$.
\end{remark}
Next, we study the decay rates of mild solutions of equation \eqref{ill-possed}. We use the following abstract lemma that extends the results of Daletskii and Krein (\cite{DK}) from the case of well-posed equations.
\begin{lemma}\label{l3.7}
Assume $G$ is the generator of an exponentially stable bi-semigroup $\{T_{\rms/\rmu}(z)\}_{z\geq0}$ on a Hilbert space $\bH$ having decay rate $-\nu$ for some $\nu>0$, $F:\RR\to\cB(\bH)$ is a piecewise strongly continuous operator valued function such that $\|F(\cdot)\|_{\cB(\bH)}\in L^1(\RR_+)\cap L^\infty(\RR_+)$. If $u$ is a mild solution of equation $u'=(G+F(z))u$ on $[a,\infty)$ satisfying the condition
$\|u(z)\|\leq Me^{\theta z}$ for any $z\geq a$, for some $\theta<\nu$, then for any $\delta\in (0,\nu)$ there exists $N>0$ such that $\|u(z)\|\leq Ne^{-\delta z}$ for any $z\geq a$.
\end{lemma}
\begin{proof}
First, we recall the following standard notations: throughout this lemma we denote by $\bH_{\rms/\rmu}$ the stable/unstable subspaces of $G$ and by $P_{\rms/\rmu}$ the projections into $\bH_{\rms/\rmu}$ parallel to $\bH_{\rmu/\rms}$. Next, we note that
\begin{equation}\label{3.7-1}
T_\rmu(\cdot-z)P_\rmu F(\cdot)u(\cdot)\in L^1([z,\infty),\bH)\quad\mbox{for any}\quad z\geq a.
\end{equation}
Indeed, one can readily check that
\begin{equation}\label{3.7-2}
\|T_\rmu(\zeta-z)P_\rmu F(\zeta)u(\zeta)\|\leq M e^{-\nu(\zeta-z)}e^{\theta\zeta}=(M e^{\nu z})e^{-(\nu-\theta)\zeta}
\end{equation}
for any $\zeta\geq z\geq a$. Since $\theta<\nu$ assertion \eqref{3.7-1} follows shortly. Since $u$ is a mild solution of equation $u'=(G+F(z))u$ on $[a,b]$ for any $b>a$, we have that
\begin{align}\label{3.7-3}
u(z)&=T_\rms(z-a)P_\rms u(a)+T_\rmu(b-z)P_\rmu u(b)+\int_a^z T_\rms(z-\zeta)P_\rms F(\zeta)u(\zeta)\rmd\zeta\nonumber\\
&\qquad-\int_z^b T_\rmu(\zeta-z)P_\rmu F(\zeta)u(\zeta)\rmd\zeta\quad\mbox{for any}\quad z\in [a,b].
\end{align}
From \eqref{3.7-2} it follows that $\lim_{b\to\infty}T_\rmu(b-z)P_\rmu u(b)=0$ for any $z\geq a$. Passing to the limit in \eqref{3.7-3}, from \eqref{3.7-1} we obtain that
\begin{align}\label{3.7-4}
u(z)&=T_\rms(z-a)P_\rms u(a)+\int_a^z T_\rms(z-\zeta)P_\rms F(\zeta)u(\zeta)\rmd\zeta\nonumber\\
&\qquad-\int_z^\infty T_\rmu(\zeta-z)P_\rmu F(\zeta)u(\zeta)\rmd\zeta\,\mbox{ for any }\, z\geq a.
\end{align}
We conclude that
\begin{equation}\label{3.7-5}
\|u(z)\|\leq Me^{-\nu(z-a)}+M\int_a^\infty e^{-\nu|z-\zeta|}\|F(\zeta)\|_{\cB(\bH)}\,\|u(\zeta)\|\rmd\zeta\,\mbox{ for any }\, z\geq a.
\end{equation}
Since $\|F(\cdot)\|_{\cB(\bH)}\in L^1(\RR_+)$ from \cite[Chapter III, Lemma 2.2]{DK} we infer that for any $\delta\in(0,\nu)$ there exists $N>0$ such that $\|u(z)\|\leq Ne^{-\delta z}$ for any $z\geq a$. \end{proof}
Next, we prove that the function $Y_0$ decays exponentially, which allows us to conclude that $v_0$ is a genuine eigenfunction.
\begin{lemma}\label{l3.8}
Assume Hypotheses (H1)-(H3), $\Re\;\lambda_0>\sup\Re\sigma_{\mathrm{ess}}(\cL)$. We recall the definition of the function $v_0$ introduced in \eqref{eigenfuntion-v0}. Then, the following assertions hold true:
\begin{enumerate}
\item[(i)] There exist $M,\delta>0$ such that
\begin{equation}\label{v0-est}
\|v_0(\cdot,z)\|_{H^1_0(\Omega)}^2+\|\pa_zv_0(\cdot,z)\|_{L^2(\Omega)}^2\leq M^2e^{-2\delta|z|}\quad\mbox{for any}\quad z\in\RR;
\end{equation}
\item[(ii)] The function $v_0$ belongs to $H^2(\Omega\times\RR)\cap H_0^1(\Omega\times\RR)$ and therefore is an eigenfunction of the linear operator $\Delta_{x,z}+V-c^2/4$.
\end{enumerate}
\end{lemma}
\begin{proof} (i) From Lemma~\ref{l3.2} we have that the linear operator $A_\pm$ generates an exponentially stable bi-semigroups on $H_0^1(\Omega)\times L^2(\Omega)$ having decay rate $-\nu_\pm$ such that $\nu_\pm>|c|/2$. Therefore, we can choose $\delta>0$ such that $\nu_\pm>\delta>|c|/2$. Since $Y_0(z)=(v_0(\cdot,z),\pa_zv_0(\cdot,z))^{\mathrm{T}}$ is a mild solution of equation $Y_0'=\big(A_++B_+(z)\big)Y_0$ by Remark~\ref{r3.6},
from Lemma~\ref{r3.3}  and Lemma~\ref{l3.7} we infer that there exists $M>0$
\begin{equation}\label{3.8-1}
\|Y_0(z)\|_{H_0^1(\Omega)\times L^2(\Omega)}\leq M e^{-\delta z}\quad\mbox{for any}\quad z\geq 0.
\end{equation}
From the definition of bi-semigroups, one can readily check that $-A_-$ generates an exponentially stable bi-semigroup on $H_0^1(\Omega)\times L^2(\Omega)$. Making the change of variables $z\to-z$ in equation $Y_0'=\big(A_-+B_-(z)\big)Y_0$, it follows that $Y_0(-\cdot)$ is a mild solution of equation $Y'=\big(-A_--B_-(-z)\big)Y$. Using again Lemma~\ref{r3.3}, Remark~\ref{r3.6} and Lemma~\ref{l3.7} we conclude that
\begin{equation}\label{3.8-1-bis}
\|Y_0(-z)\|_{H_0^1(\Omega)\times L^2(\Omega)}\leq M e^{-\delta z}\,\mbox{ for any }\, z\geq 0.
\end{equation}
The estimate \eqref{v0-est} follows shortly from \eqref{3.8-1} and \eqref{3.8-1-bis}.

(ii) Since $u_0\in H^2(\Omega\times\RR)\cap H^1_0(\Omega\times\RR)$ and $v_0(x,z)=e^{cz/2}u_0(x,z)$ for any $x\in\Omega$ and $z\in\RR$, from \eqref{v0-est} we immediately conclude that $v_0\in H^1_0(\Omega\times\RR)$. To prove (ii) we use that $v_0$ satisfies equation \eqref{eigenfuntion-v0}. We introduce $\kappa_n\in C_0^\infty(\RR)$ such that $0\leq\kappa_n\leq 1$, $\kappa_n(z)=1$ for any $z\in [-n,n]$, $\supp\kappa_n$ is compact and the derivatives of $\kappa_n$ satisfy
\begin{equation}\label{kappa-deriv}
\sup_{n\geq1}\|\kappa_n^{(j)}\|_\infty<\infty\quad\mbox{for}\quad j=1,2.
\end{equation}
Let $v_n:\Omega\times\RR\to\CC$ be the function defined by $v_n(x,z)=\kappa_n(z)v_0(x,z)$. Since $v_n(x,z)=e^{cz/2}\kappa_n(z)u_0(x,z)$ for any $n\geq 1$, $x\in\Omega$, $z\in\RR$, $\kappa_n\in C_0^\infty(\RR)$ for any $n\geq 1$ and $u_0\in H^2(\Omega\times\RR)\cap H_0^1(\Omega\times\RR)$ we obtain that $v_n\in H^2(\Omega\times\RR)\cap H_0^1(\Omega\times\RR)$ and
\begin{equation}\label{3.8-2}
(\Delta v_n)(x,z)=2\kappa_n'(z)\pa_zv_0(x,z)+\Big(\kappa_n''(z)-\big(\frac{c^2}{4}+\lambda_0-V(x,z)\big)\kappa_n(z)\Big)v_0(x,z)
\end{equation}
for any $n\geq 1$, $x,z\in\RR$. Since $v_0\in H_0^1(\Omega\times\RR)$, from (i), \eqref{kappa-deriv} and \eqref{3.8-2} it follows that $\sup_{n\geq 1}\|(\Delta_{x,z}-I)v_n\|_{L^2(\Omega\times\RR)}<\infty$. Using that $\Delta_{x,z}-I$ is invertible with bounded inverse from $H^2(\Omega\times\RR)\cap H_0^1(\Omega\times\RR)$ to $L^2(\Omega\times\RR)$ we obtain that $\sup_{n\geq 1}\|v_n\|_{H^2(\Omega\times\RR)\cap H_0^1(\Omega\times\RR)}<\infty$.
Since $\kappa_n(z)=1$ for any $n\geq 1$ and $z\in[-n,n]$, we conclude that $v_0\in H^2(\Omega\times\RR)\cap H_0^1(\Omega\times\RR)$, proving (ii) and the lemma.
\end{proof}

\noindent\textbf{Proof of Theorem~\ref{t1.2}.} Fix $\lambda_0\in\sigma_{\rmd}(\cL)$ such that $\mathrm{Re}\,\lambda_0>\sup\mathrm{Re}\,\sigma_{\mathrm{ess}}(\cL)$. From \eqref{eigenfuntion-v0} and Lemma~\ref{l3.8}(ii) we have that $\lambda_0\in\sigma_{\rmd}(\Delta_{x,z}+V-c^2/4)$. Since the linear operator $\Delta_{x,z}+V-c^2/4$ is self-adjoint we conclude that $\lambda_0\in\RR$, proving the theorem.

\section{Discussion: the finite-dimensional case}\label{s4}
In this section we describe how to use the methods from this paper to recover the results from \cite{BoJo,HLS2} similar to Theorem~\ref{t1.2}.
Let $\ou(y-ct)$ be a traveling wave of the Allen-Cahn equation
\begin{equation}\label{5.9-1}
u_t+F'(u)=u_{yy},\quad t\geq 0,\;y\in\RR,
\end{equation}
and assume its profile is such that $u_\pm=\lim_{z\to\pm\infty}\ou(z)$ exists. The linearization along the wave is given by $\cL_*=\pa_z^2+c\pa_z-F''(\ou(z))$ in the moving frame $z=y-ct$.
\begin{theorem}(\cite{BoJo,HLS2})\label{t4.1}
Assume that the profile $\ou$ satisfies the condition
\begin{equation}\label{assumption-HLS2}
\int_{-\infty}^s (1+|x|)|\ou(x)-u_-|\rmd x<\infty\;\mbox{and}\;\int_{s}^\infty (1+|x|)|\ou(x)-u_+|\rmd x<\infty
\end{equation}
for any $s\in\RR$. Then, the discrete spectrum of $\cL_*$ to the right of the right most point of its essential spectrum is real.
\end{theorem}

To prove this result, the authors make the change of variables $v_*(z)=e^{-c/2z}u_*(z)$ in the eigenvalue equation $\cL_*u_*=\lambda_0 u_*$, where $\lambda_0\in\sigma_\rmd(\cL_*)$ with $\Re\lambda_0>\sup\Re\;\sigma_{\mathrm{ess}}(\cL_*)$. Then, they control the decay rate of the eigenfunction $u_*$ by computing the decay rate of the Jost solutions associated to the eigenvalue problem. This argument allows then to conclude that $v_*$ is a genuine eigenfunction of the self-adjoint operator $\pa_z^2-c^2/4-F''(\ou(\cdot))$ associated to the eigenvalue $\lambda_0$, which proves the statement.

The key part of this argument is to control the decay rate of the eigenfunction $u_*$. In the infinite-dimensional case its not straitforward to construct the Jost solutions of the not well-posed, first order equation \eqref{eigenfuntion-firstorder}, nor to conclude from here that the decay rate of the eigenfunction $u_0$ introduced in \eqref{eigenfuntion-u0} can be evaluated using the decay rates of the (infinitely many) Jost solutions. In fact, in \cite{LP3} it is shown that the construction of Jost solutions associated to \eqref{eigenfuntion-firstorder} requires a significant effort in this infinite-dimensional, not well-posed case.
The method we used to prove Theorem~\ref{t1.2} can be applied to control the growth rate of the eigenfunction $u_*$ by recasting the second order equation
\begin{equation}\label{5.9-2}
v_*''-\big(\frac{c^2}{4}+F''(\ou(z))\big)v_*=\lambda_0v_*
\end{equation}
as a first order, linear differential equation on some finite-dimensional space, whose decay rates can be estimated using Lemma~\ref{l3.7}. Moreover, we can prove the result by assuming that $\ou-u_\pm\in L^1(\RR_\pm)\cap L^\infty(\RR_\pm)$, thus relaxing assumption \eqref{assumption-HLS2}.
Indeed, as shown in \cite{HLS2}, if $w_*=v_*'$ then the pair $(v_*,w_*)$ satisfies the equation
\begin{equation}\label{5.9-3}
\frac{d}{dz}\begin{pmatrix}v_*\\w_*\end{pmatrix}=A_*(z)\begin{pmatrix}v_*\\w_*\end{pmatrix},
\end{equation}
where $A_*(z)$ is the matrix-valued function defined by
\begin{equation}\label{5.9-4}
A_*(z)=\begin{bmatrix}
0 & I \\
\lambda_0+\frac{c^2}{4}+F''(\ou(z))& 0
\end{bmatrix}.
\end{equation}
Similar to the infinite-dimensional case \eqref{eigenfuntion-firstorder}, we have the decomposition $A_*(z)=A_{*,\pm}+B_{*,\pm}(z)$, where
\begin{equation}\label{5.9-5}
A_{*,\pm}=\begin{bmatrix}
0 & I \\
\lambda_0+\frac{c^2}{4}+F''(u_\pm)& 0
\end{bmatrix}, \; B_{*,\pm}(z)=\begin{bmatrix}
0 & 0 \\
F''(u_\pm)-F''(\ou(z))& 0
\end{bmatrix}.
\end{equation}
Similar to the proof of Lemma~\ref{l3.2}, using Lemma~\ref{l3.1} and the condition $\Re\lambda_0>\sup\Re\;\sigma_{\mathrm{ess}}(\cL_*)$, we can show that $A_{*,\pm}$ is a hyperbolic matrix, hence it generates a bi-semigroup. Since $F$ is a smooth function (typically one assumes that $F$ is of class $C^2$) and $\ou-u_\pm\in L^1(\RR_\pm)\cap L^\infty(\RR_\pm)$ it follows that $\|B_{*,\pm}(\cdot)\|\in L^1(\RR_\pm)\cap L^\infty(\RR_\pm)$. Next, we can apply Lemma~\ref{l3.7} to prove that
$\|(v_*(z),w_*(z))^{\mathrm{T}}\|\leq M e^{-\delta |z|}$ for any $z\in\RR$, for some $M,\delta>0$. From here we can immediately conclude that $v_*$ is a genuine eigenfunction of the self-adjoint operator $\pa_z^2-c^2/4-F''(\ou(\cdot))$ associated to the eigenvalue $\lambda_0$, proving that $\lambda_0$ is real.

\end{document}